\documentclass[11pt]{article}

\usepackage[margin=1in]{geometry}
\usepackage[T1]{fontenc}
\usepackage{lmodern,comment}
\usepackage{amsmath,amssymb,amsthm,mathtools}
\usepackage{microtype}
\usepackage{enumitem}
\usepackage[hidelinks]{hyperref}
\usepackage[numbers,sort&compress]{natbib}
\newtheorem{theorem}{Theorem}[section]
\newtheorem{lemma}[theorem]{Lemma}
\newtheorem{corollary}[theorem]{Corollary}
\newtheorem{conjecture}{Conjecture}
\theoremstyle{remark}

\numberwithin{equation}{section}

\DeclareMathOperator{\ex}{ex}
\DeclareMathOperator{\Perm}{Perm}
\DeclareMathOperator{\supp}{supp}
\newcommand{\sh}{\partial}
\newcommand{\Wcal}{\mathcal W}
\newcommand{\Scal}{\mathcal S}
\hypersetup{
	pdftitle={Kalai's Conjecture for tight trees},
	pdfauthor={Dhruv Mubayi},
	pdfsubject={A permutation proof of the shadow bound for tight trees}
}

\title{Kalai's Conjecture for Tight Trees}
\author{
	Dhruv Mubayi\thanks{Department of Mathematics, Statistics and Computer Science, University of Illinois, Chicago, IL 60607. Email: mubayi@uic.edu. Research partially supported by NSF Awards DMS-2552740 and DMS-2153576.} \and
	Jacques Verstra\" ete\thanks{Department of Mathematics, University of California, San Diego, CA, 92093-0112 USA.
		Email: jverstraete@ucsd.edu.
		Research supported by NSF-BSF award DMS-2347832.  }}
\date{}

\begin{document}
	\maketitle
	
	\begin{abstract}
		Let $r \ge 2$ and $t \ge 1$. It is shown that if $T$ is an $r$-uniform tight tree with $t$ edges and
		$H$ is a $T$-free $r$-uniform hypergraph, then
		$|E(H)|\le (t-1)|\sh H|/r$, where $\sh H$ is the $(r-1)$-shadow
		of $H$. \iffalse Equality holds only for $(n,t + r - 2,r)$-designs.\fi The bound is tight infinitely often, and establishes Kalai's Conjecture, whose $r=2$ case is the Erd\H os-S\'os Conjecture.
		
		The proof was
		found by GPT-6 Astra, extending its method of proof for the Erd\H os-S\'os conjecture to the hypergraph setting. It is noteworthy that previous proofs of special cases of the Erd\H os-S\'os conjecture do not extend to give tights bounds in the hypergraph setting. 
		
		A strengthening of the Erd\H os-S\'os conjecture due to the authors about tight lower bounds on the number  of copies of a tree in a graph with  average degree $d \ge t-1\ge 0$ remains open. 
	\end{abstract}
	
	\section{AI Declaration Statement}
	The proof of the main result was found by GPT-6 Astra by first  prompting it to extend its recent proof of the Erd\H os-S\'os conjecture (see Adamczewski and Bloom~\cite[Appendix B.4]{Epoch})  to the case of hypergraph tight paths as in~\cite{FJKMVpaths}, and then prompting it to prove the full Kalai conjecture. As we elaborate throughout the text, GPT-6 Astra's  proof technique for Erd\H os-S\'os conjecture, which uses permutations of the underlying vertex set, perhaps has its genesis in an old argument of Perles (see~\cite{KupitzPerles}), which considered convex geometric graphs (these corresponds to cyclic permutations of the vertex set).  This method was developed in the hypergraph case by F\"uredi-Jiang-Kostochka-Mubayi-Verstra\"ete in~\cite{FJKMVpaths} and gave the previously best bounds for Kalai's conjecture for tight paths. The authors have checked the proof and rewritten much of the text provided by GPT-6 Astra and added further explanations throughout the paper. They take full responsibility for all the content in the paper.

	\section{Introduction}\label{sec:introduction}
	
	All hypergraphs in this paper are finite and simple. An $r$-uniform
	hypergraph, or \emph{$r$-graph}, has edges of size $r$. For an $r$-graph
	$T$, let $\ex(n,T)$ denote the maximum number of edges in an $n$-vertex
	$r$-graph containing no copy of $T$. 
	
	The Erd\H{o}s--Gallai Theorem~\cite{EG} asserts that a graph with no
	path of $t$ edges has at most $(t-1)n/2$ edges. The
	Erd\H{o}s--S\'os Conjecture~\cite{Erdos} asks for the same bound when
	the forbidden path is replaced by an arbitrary tree with $t$ edges. This conjecture has received extensive attention in the literature~\cite{ASS,BPSS2019,BPSS2021,BD,DPRS,Erdos,Erdos1964,EG,ErdosSos1970,FanSun2007,GoerlichZak2016,HRSW2020,McLennan2005,Pokrovskiy2024,Rozhon2019,ReedStein2026,SacleWozniak1997,Stein,TinerTomlin2022,Wozniak1996}. The Erd\H{o}s--S\'os Conjecture was recently proved employing a permutation argument by GPT-6 Astra AI, described by Adamczewski and Bloom~\cite[Appendix B.4]{Epoch}. Similar counting arguments are key to the method of Perles~\cite{KupitzPerles} giving a very elegant alternative proof of the Erd\H{o}s--Gallai Theorem using convex geometric graphs, and to extensions to uniform hypergraphs in~\cite{FJKMVpaths}.

	Kalai proposed a hypergraph generalization of the Erd\H{o}s--S\'{o}s Conjecture in 1984, recorded by
	Frankl and F\"uredi~\cite{FF}, using the following notion of a tree. An $r$-graph $T$ is a \emph{tight $r$-tree} if its edges have an ordering
	$e_1,\ldots,e_t$ such that for every $i>1$, there exists
	$z_i\in e_i$ and an index $p(i)<i$ satisfying
	\begin{equation}\label{eq:tight-tree}
		z_i\notin\bigcup_{h<i}e_h,
		\qquad e_i\setminus\{z_i\}\subseteq e_{p(i)}.
	\end{equation}
	We take $V(T)=\bigcup_{e\in E(T)}e$, so a tight $r$-tree with $t$
	edges has exactly $t+r-1$ vertices -- the case $r = 2$ is precisely a tree in the graph sense. Kalai's Conjecture is as follows:
	
	\begin{conjecture}\label{conj:kalai} {\bf (Kalai's Conjecture)} 
		For $n \geq r \geq 2$ and every tight $r$-tree $T$ with $t$ edges,
		\begin{equation}\label{eq:kalai-binomial}
			\ex(n,T)\le\frac{t-1}{r}\binom{n}{r-1}.
		\end{equation}
	\end{conjecture} 
	
	To motivate this conjecture, let $b = t + r - 2$ and take a family of $b$-vertex blocks in which each $(r-1)$-set
	lies in at most one block and almost every $(r-1)$-set is covered,
	and put a complete $r$-graph on each block. Any copy of a tight tree
	must stay in a single block -- two edges from different blocks cannot
	meet in $r-1$ vertices -- and therefore each tight tree has at most $b = t + r - 2$ vertices and therefore fewer than $t$ edges. Using
	R\"odl's packing theorem~\cite{Rodl} providing the existence of asymptotically optimal partial Steiner systems, one obtains as $n \rightarrow \infty$:
	\[
	\ex(n,T)\ge\left(\frac{t-1}{r}-o(1)\right) \cdot \binom{n}{r-1}.
	\]
	Keevash's existence theorem for designs~\cite{Keevash} (see also Glock, K\"{u}hn, Lo and Osthus~\cite{GKLO}) gives exact
	equality in~\eqref{eq:kalai-binomial} for infinitely many $n$ by
	covering every $(r-1)$-set exactly once with $b$-sets -- these are Steiner systems $S(r-1,b,n)$. In the special case that the tree $T$ is a sunflower -- every edge of $T$ has the form $R \cup \{v\}$ where $|R| = r - 1$ -- then any $r$-uniform hypergraph in which every $(r - 1)$-set is contained in exactly $t - 1$ hyperedges is $T$-free. These are sometimes called \emph{simple $(r - 1)$-$(n,r,t-1)$-designs}.

	Several cases of Kalai's Conjecture were previously known. Frankl
	and F\"uredi~\cite{FF} proved (\ref{eq:kalai-binomial}) for star-shaped tight
	trees, which have an edge meeting every other edge in $r-1$ vertices.
	F\"uredi, Jiang, Kostochka and the authors~\cite{FJKMVtrunk}
	proved an asymptotic version for tight trees with a bounded trunk and
	the exact bound for tight trees with at most four edges. Here a trunk
	is an initial tight subtree in a construction~\eqref{eq:tight-tree}
	such that each remaining edge attaches directly to an edge of that
	subtree; the asymptotic statement fixes $r$ and the trunk size and
	lets $t$ tend to infinity. Stein~\cite{Stein} proved (\ref{eq:kalai-binomial})
	for all tight trees when the host $r$-graph is $r$-partite.
	
	\subsection{Main Theorem}
	
	The \emph{shadow} of an $r$-graph $H$ is
	\[
	\sh H=\{e\subseteq V(H):|e|=r-1,\ e\subseteq h
	\text{ for some }h\in E(H)\}.
	\]
	We prove  Kalai's
	conjecture by proving the following shadow bound  (the equivalence to Kalai's Conjecture was established by F\"uredi, Jiang, Kostochka and the authors~\cite[Proposition 2.2]{FJKMVtrunk}).
	
	\begin{theorem}\label{thm:main}
		Let $n \ge r\ge 2$, and let $T$ be a tight $r$-tree with $t\ge1$ edges\iffalse, and $b = t + r - 2$\fi.
		Every $n$-vertex $T$-free $r$-graph $H$ satisfies
		\begin{equation}\label{eq:main-shadow}
			|E(H)| \; \le \; \frac{t-1}{r} \cdot |\sh H| \; \le \; \frac{t-1}{r} \cdot \binom{n}{r-1} .
		\end{equation}
		
	\end{theorem}

	This theorem also settles Conjecture 1.1 of  F\"uredi, Jiang, Kostochka and the authors~\cite{FJKMVpaths}. The proof is given in Section \ref{sec:trees}. 
	
	GPT-6 Astra has also give a proof that equality holds in Theorem~\ref{thm:main} if and only if $H$ is a Steiner $S(r-1,b,n)$-system ($b=t+r-2$) or $T$ is a sunflower and every $(r - 1)$-subset of $V(H)$ is contained in exactly $t - 1$ edges of $H$. We have not verified this claim regarding the equality characterization. 
	
	\subsection{Counting conjecture}
	
	The following generalization of the Erd\H{o}s-S\'{o}s Conjecture was stated by the authors in~\cite{MubayiV2016}: let $T$ be a tree with $t$ edges. Then there exists $d_0=d_0(t)$ such that every $n$-vertex graph $G$ of average degree $d\geq d_0$ contains at least \begin{equation}\label{eq:treecount} N_T(G) \geq n(d)_t = n d(d-1)(d-2)\cdots(d-t+1) \end{equation} labeled copies of $T$, where $N_T(G)$ denotes the number of injective graph homomorphisms from $T$ to $G$. 
	The authors in~\cite{MubayiV2016} proved the asymptotically sharp estimate \begin{equation} 
		N_T(G) \geq \left(1-O\left(\frac{t^5}{d^2}\right)\right)n(d)_t. \end{equation} 
	The conjectured bound (\ref{eq:treecount}) has subsequently been proved by Wilson~\cite{Wilson2025} for $d_0=\Omega(t^4)$. A natural strengthening posed by the authors would be to prove the result with \begin{equation} d_0(t) = t - 1. \end{equation} Such a statement would clearly imply the Erd\H{o}s--S\'os conjecture.  It is not too difficult to verify that the conjecture holds for certain small trees, as well as for double stars.
	A proof, which we have not verified, is claimed by GPT-6 Astra when $d=\Omega(t^3)$ for all trees $T$ and when $d=\Omega(t^2)$ and $T$ is a  path with $t$ edges.

	We make the following 
	even more general hypergraph counting conjecture:
	
	\begin{conjecture} \label{conj:hypcount}
		Let $n \geq r \geq 2$, let $d \geq t - 1 \geq 0$, and let $H$ be an $n$-vertex $r$-uniform hypergraph with at least $\frac{d}{r}{n \choose r - 1}$ edges. Then for any tight tree $T$ with $t$ edges, 
		\begin{equation}
			N_T(H) \geq (r-1)!{n\choose r-1}(d)_t.
		\end{equation}    
	\end{conjecture}
	\bigskip
	
	In other words, any hypergraph of average $(r - 1)$-degree $d \geq t - 1$ contains at least $(r-1)!{n\choose r-1}(d)_t$ labelled copies of any prescribed tight tree $T$ with $t$ edges. This implies Kalai's Conjecture (Conjecture \ref{conj:kalai}), but, as mentioned above, remains open even for paths in graphs. GPT-6 Astra has claimed a proof of Conjecture~\ref{conj:hypcount} for $d>16r^2t^3$. We have not verified this. 
	
	\subsection{Organization} 
	
	The special case of trees that are {\em tight paths} contains the main ideas of the proof of Theorem \ref{thm:main}, and is presented in Section \ref{sec:paths}. 
	This proof generalizes the ideas in~\cite{FJKMVpaths}, which previously had the best bound for the extremal number of tight paths. The proof of Theorem \ref{thm:main} is given in subsequent sections.
	\iffalse
	and, in Section \ref{sec:equality}, we characterize the case of equality.
	\fi
	
	\section{Proof of Theorem \ref{thm:main} for tight paths}\label{sec:paths}
	
	The \emph{tight path} $P_k^r$ has vertex set
	$\{x_1,\ldots,x_{k+r-1}\}$ and edges
	$\{x_i,\ldots,x_{i+r-1}\}$ for $1\le i\le k$. Using convex geometric hypergraphs, F\"uredi, Jiang, Kostochka and the authors~\cite[Theorem 1.2]{FJKMVpaths} proved
	\begin{equation}\label{eq:previous-path}
		\ex(n,P_k^r)\le
		\begin{cases}
			\displaystyle\frac{k-1}{2}\binom{n}{r-1},&r\text{ even},\\[6pt]
			\displaystyle\frac12\left(k+\left\lfloor\frac{k-1}{r}\right\rfloor\right)
			\binom{n}{r-1},&r\text{ odd}.
		\end{cases}
	\end{equation}
	
	The proof of this theorem, which proceeds by considering a count of ordered tight paths in so-called convex geometric hypergraphs~\cite{FJKMVpaths} partly motivates the proof of Theorem~\ref{thm:main}; indeed in this section we prove the tight path case of Theorem \ref{thm:main}, as it is much simpler and yet illustrates the main ideas. Roughly speaking, we count permutations whose first $r-1$ vertices, together with a marked later vertex, form an edge of the host hypergraph. For paths, a cyclic rotation extends every eligible state except the first one in each permutation, and these ideas were used to study the extremal function for tight paths in convex geometric hypergraphs~\cite{FJKMVpaths}, and we give this simpler ``path argument'' proof of Theorem \ref{thm:main} in this section. The more complicated case of general trees is achieved in Sections \ref{subsec:rooted-states} -- \ref{subsec:operationsroot}.

	\subsection{Definitions and notation} \label{subsection:defns} Fix an $r$-graph $H$ on a vertex set $V$ of size $n\ge r$.
	For a permutation $\pi=(v_1,\ldots,v_n)$ of $V$, write
	\[
	R(\pi)=(v_1,\ldots,v_{r-1}),\qquad
	V_j(\pi)=\{v_1,\ldots,v_j\}.
	\]
	Let $\Wcal$ be the set of permutations whose first $r-1$ vertices
	form a member of $\sh H$, and define the set of \emph{marked states}
	by
	\begin{equation}\label{eq:states}
		\Omega=\bigl\{(\pi,j):\pi\in\Wcal,\ r\le j\le n,
		\ \{v_1,\ldots,v_{r-1},v_j\}\in E(H)\bigr\}.
	\end{equation}
	Set $W=|\Wcal|$ and $M=|\Omega|$. Direct counting gives
	\begin{equation}\label{eq:counts}
		\begin{aligned}
			W&=(r-1)!\,|\sh H|\,(n-r+1)!,\\
			M&=r!\,|E(H)|\,(n-r+1)!.
		\end{aligned}
	\end{equation}
	For the first identity, choose the initial shadow edge, order it,
	and order the remaining vertices. For the second, choose an edge,
	choose its marked vertex in $r$ ways, and order its other $r-1$
	vertices at the beginning of the permutation. Order all remaining
	$n-r+1$ vertices arbitrarily; the position of the chosen marked
	vertex determines $j$. For $\ell\ge1$, let $\Scal_\ell\subseteq\Omega$ consist of the
	states $(\pi,j)$ for which $H[V_j(\pi)]$ contains a tight path of
	$\ell$ edges ending with the ordered tuple $R(\pi)$. Explicitly,
	there must be distinct vertices $x_1,\ldots,x_{\ell+r-1}\in V_j(\pi)$
	such that
	\begin{equation}\label{eq:path-state}
		\mbox{For }1 \leq i \leq \ell, 
		\{x_i,\ldots,x_{i+r-1}\} \in E(H)\mbox{ and }
		(x_{\ell+1},\ldots,x_{\ell+r-1}) = R(\pi).
	\end{equation}
	The permutation specifies the allowed vertex set and the terminal
	tuple; the path may traverse its other vertices in any order.
	Every marked state supports the one-edge path
	$(v_j,v_1,\ldots,v_{r-1})$, so
	\begin{equation}\label{eq:path-base}
		\Scal_1=\Omega.
	\end{equation}
	
	\subsection{Mapping lemma for tight paths}
	
	\begin{lemma}\label{lem:path-transfer}
		For every $\ell\ge1$,
		\begin{equation}\label{eq:path-transfer}
			|\Scal_\ell|\le |\Scal_{\ell+1}|+W.
		\end{equation}
	\end{lemma}
	
	\begin{proof}
		For each $\pi\in\Wcal$ having a state in $\Scal_\ell$, discard
		$(\pi,j)$ with $j$ least among its states in $\Scal_\ell$.
		At most $W$ states are discarded. We provide an injection $\Phi'$ from the remaining states
		into $\Scal_{\ell+1}$.
		
		For any $(\pi,j)\in\Omega$, write
		\[
		\pi=(a_1,\ldots,a_{r-1},A,x,B),\qquad x=v_j,
		\]
		where $A=(v_r,\ldots,v_{j-1})$ and $B=(v_{j+1},\ldots,v_n)$
		are possibly empty words. Define
		\begin{equation}\label{eq:path-rotation}
			\Phi(\pi,j)=(\pi',j),\qquad
			\pi'=(a_2,\ldots,a_{r-1},x,A,a_1,B).
		\end{equation}
		This rotates the entries in positions $1,\ldots,r-1,j$ and fixes
		all other entries. It preserves the marked edge, and hence maps
		$\Omega$ to itself. Since $\Phi^r$ is the identity, $\Phi$ is a
		bijection of $\Omega$. Moreover,
		\begin{equation}\label{eq:path-support}
			V_j(\pi')=V_j(\pi),\qquad
			R(\pi')=(a_2,\ldots,a_{r-1},x).
		\end{equation}
		
		Now let $(\pi,j)\in\Scal_\ell$ be a state that has not been discarded in the first step. There exists
		$i<j$ with $(\pi,i)\in\Scal_\ell$. A witnessing path for this
		earlier state lies in $V_i(\pi)$ and ends with
		$(a_1,\ldots,a_{r-1})$. Since $x=v_j\notin V_i(\pi)$, appending
		$x$ introduces a new vertex and the new edge
		$\{a_1,\ldots,a_{r-1},x\}\in E(H)$.
		The resulting tight path has $\ell+1$ edges, lies in $V_j(\pi')$,
		and ends with $R(\pi')$. Thus $\Phi(\pi,j)\in\Scal_{\ell+1}$.
		The restriction of the bijection $\Phi$ is the required injection $\Phi'$.
	\end{proof}
	
	\begin{proof}[Proof of Theorem~\ref{thm:main} for tight paths]
		We aim to show that if $H$ is an $n$-vertex $r$-uniform $P_k^r$-free hypergraph, then $|E(H)| \leq (k - 1)/r \cdot |\sh H|$. The assertion is immediate if $E(H)=\varnothing$. Otherwise
		$n\ge r$, so the preceding definitions apply. If $H$ is
		$P_k^r$-free, then $\Scal_k=\varnothing$. Iterating
		Lemma~\ref{lem:path-transfer} and using~\eqref{eq:path-base} gives
		\[
		M=|\Scal_1|\le (k-1)W.
		\]
		Substituting~\eqref{eq:counts} and cancelling the factor
		$(r-1)!(n-r+1)!$ yields $r|E(H)|\le(k-1)|\sh H|$. \end{proof}
	
	\section{Proof of Theorem \ref{thm:main}}\label{sec:trees}
	
	For general tight trees, we prescribe the image of an entire root edge as in Adamczewski and Bloom~\cite[Appendix B.4]{Epoch}. A prefix permutation inequality then allows two rooted pieces to be joined
	without an additional error term. This part is not necessary in the proof given in Section \ref{sec:paths}, as we may repeatedly remove leaf edges and map them to leaf edges of shorter tight paths, whereas for a tight tree, we cannot map leaf edges to leaf edges.
	Induction then combines this joining
	operation with the addition of a leaf edge.

	We retain the notation $H$, $V$, $\Wcal$, $\Omega$, $W$ and $M$ from
	Section~\ref{sec:paths}. The proof for trees uses embeddings with a
	prescribed image of a whole root edge. 
	The following straightforward decomposition lemma allows us to do induction on the size of the tree.
	
	\begin{lemma}\label{lem:structure}
		Let $T$ be a tight $r$-tree with at least two edges, and fix
		$e\in E(T)$. At least one of the following holds.
		\begin{enumerate}[label=\textup{(\roman*)},leftmargin=*]
			\item\label{case:leaf}
			$T$ is obtained from a smaller tight tree $S$ by adding a new
			vertex $z$ and the edge $e=F\cup\{z\}$, where $F$ is an
			$(r-1)$-subset of an edge $f$ of $S$.
			\item\label{case:gluing}
			$T=T_1\cup T_2$ for tight trees $T_1,T_2$, each with fewer edges
			than $T$, such that
			\[
			E(T_1)\cap E(T_2)=\{e\},\qquad V(T_1)\cap V(T_2)=e.
			\]
		\end{enumerate}
	\end{lemma}
	
	This lemma is necessary for our inductive proof, which maps leaves to leaves when $T$ is a tight path (as happens in Case (i)), but which may map leaves to non-leaf edges when $T$ is not a path (Case (ii)). We omit the proof of Lemma~\ref{lem:structure}, which is an easy exercise.

	\subsection{States rooted at an edge}\label{subsec:rooted-states}
	We now generalize the definition of ${\mathcal S}_{\ell}$ in Section~\ref{subsection:defns} from tight paths to arbitrary hypergraphs.
	Let $S$ be an $r$-graph with a distinguished ordered edge
	$\vec e=(u_1,\ldots,u_r)$. For $V(H)=\{v_1, \ldots, v_n\}$ and $\pi=v_1v_2\ldots v_n$, a state $(\pi,j)\in\Omega$ is
	\emph{good for $(S,\vec e)$} if there is an injective embedding
	$\varphi:V(S)\longrightarrow V(H)$ such that for $1 \leq i \leq r-1$,
	\begin{equation}\label{eq:rooted-state}
		\begin{gathered}
			\varphi(u_i)=v_i,\qquad \varphi(u_r)=v_j,\\
			\varphi(V(S))\subseteq V_j(\pi).
		\end{gathered}
	\end{equation}
	Here an embedding sends every edge of $S$ to an edge of $H$.
	Let $C(S,\vec e)$ be the number of good states. Each state is
	counted once, regardless of how many embeddings witness it.
	
	This number does not depend on the ordering of $e$. To see this,
	write $q_i=i$ for $i<r$ and $q_r=j$. If the root is reordered as
	$(u_{\sigma(1)},\ldots,u_{\sigma(r)})$, permute the host entries
	in these positions so that
	\[
	v'_{q_i}=v_{q_{\sigma(i)}}\qquad(1\le i\le r),
	\]
	and keep all other entries fixed. This is a bijection on $\Omega$,
	preserves $V_j(\pi)$, and makes the same embedding a witness for
	the reordered root. We therefore write $C(S,e)$, choosing an
	ordering of the fixed edge $e$ when convenient. The count can
	depend on the choice of $e$. For the one-edge tree,
	\begin{equation}\label{eq:rooted-base}
		C(e,e)=M.
	\end{equation}
	
	\subsection{Adding a leaf}\label{subsec:operationsleaf}
	
	The following lemma deals with Case (i) of Lemma~\ref{lem:structure} in the decomposition of a tight tree, and is the analog of the case of paths (Lemma~\ref{lem:path-transfer}) in Section \ref{sec:paths}.
	
	\begin{lemma}\label{lem:leaf}
		Let $S$ be an $r$-graph, let $f\in E(S)$, and let $F\subset f$
		have size $r-1$. Form $T$ by adding a new vertex $z$ and the
		edge $e=F\cup\{z\}$. Then
		\begin{equation}\label{eq:leaf}
			C(S,f)\le C(T,e)+W.
		\end{equation}
	\end{lemma}
	
	\begin{proof}
		Write $F=\{u_1,\ldots,u_{r-1}\}$ and $f=F\cup\{y\}$.
		Choose the root orders $(u_1,\ldots,u_{r-1},y)$ for $S$ and
		$(u_1,\ldots,u_{r-1},z)$ for $T$.
		For every $\pi\in\Wcal$, discard its first good state for
		$(S,f)$, if one exists. This discards at most $W$ states.

		For a retained state $(\pi,j)$, there is an earlier good state
		$(\pi,i)$ with $i<j$. Choose its witnessing embedding $\varphi$
		of $S$. It maps $u_h$ to $v_h$ for $h<r$, maps $y$ to $v_i$,
		and has image in $V_i(\pi)$. Since $v_j\notin V_i(\pi)$, extend
		$\varphi$ by sending $z$ to $v_j$. The additional edge maps to
		$\{v_1,\ldots,v_{r-1},v_j\}\in E(H)$, and the resulting embedding
		satisfies~\eqref{eq:rooted-state} for $(T,e)$ at $(\pi,j)$.
		Thus every retained state is itself good for $(T,e)$, proving
		\eqref{eq:leaf}.
	\end{proof}

	\subsection{Joining along root edges}\label{subsec:operationsroot}

	%\subsection{A permutation-prefix inequality}\label{subsec:prefix}
	
	The more complicated decomposition offered by Case (ii) of Lemma~\ref{lem:structure} is the subject of this section.
	For a finite set $U$, let $\Perm(U)$ denote the set of
	permutations of $U$. For $w=(w_1,\ldots,w_m)\in\Perm(U)$,
	write
	\[
	w_{\le s}=(w_1,\ldots,w_s),
	\qquad
	w_{>s}=(w_{s+1},\ldots,w_m).
	\]
	The support of a permutation or one of its blocks is
	the set of its entries. When $|U|=m$ and $\mathcal F\subseteq2^U$, put
	\[
	C_U(\mathcal F)=\{(w,s):w\in\Perm(U),\ 0\le s\le m,
	\ \supp(w_{\le s})\in\mathcal F\}\]  and  $c_U({\cal F})=|C_U({\cal F})|$.
	The empty prefix is included in $C_U(\mathcal F)$. Thus the total number of pairs
	$(w,s)$ is $(m+1)m!=(m+1)!$. 
	
	In Lemma~\ref{lem:prefix} below, we should think of a tight tree $T$ being expressed as the union of two tight trees $T_1$ and $T_2$ with a common edge $e$ so  $T=T_1 \cup T_2$ and $V(T_1) \cap V(T_2)=e$, 
	and $\cal A$, $\cal B$, $\cal C$ as vertex sets of the underlying hypergraph that support $T_1, T_2$, and $T$, respectively, where (the image of) $e$ is excluded.
	\begin{lemma}\label{lem:prefix}
		Suppose $\mathcal A,\mathcal B,\mathcal C\subseteq2^U$ satisfy
		$\mathcal C\subseteq\mathcal A$ and
		\begin{equation}\label{eq:disjoint-gluing}
			R\in\mathcal A,\quad X\in\mathcal B,\quad R\cap X=\varnothing
			\quad\Longrightarrow\quad R\cup X\in\mathcal C.
		\end{equation}
		If $|U|=m$, then
		\begin{equation}\label{eq:prefix-gluing}
			c_U(\mathcal A)+c_U(\mathcal B)
			\le (m+1)!+c_U(\mathcal C).
		\end{equation}
	\end{lemma}
	
	\begin{proof}
		We construct an injection from the set of pairs $(w,s)\in  
		C_U(\mathcal A)$  for which $\supp(w_{\le s})\not \in\mathcal C$ to the set of all pairs $(w', s')$ with  $\supp(w'_{\le s'})\not \in\mathcal B$. 
		The will imply
		$c_U(\mathcal A)-c_U(\mathcal C)
		\le (m+1)!-c_U(\mathcal B)$ as required.
		
		Let $(w,s)$ be a pair of the first
		kind, and let $R$ be the shortest prefix of $w$ with
		$\supp(R)\in\mathcal A\setminus\mathcal C$.
		This prefix exists and has length at most $s$. Write
		\[
		w=RXY,\qquad s=|R|+|X|,\qquad w'=XRY, \qquad s'=|X|
		\]
		where juxtaposition denotes concatenation, and define
		\begin{equation}\label{eq:prefix-map}
			(w,\,s)\longmapsto(w',\,s').
		\end{equation}
		If $\supp(X)\in\mathcal B$, then~\eqref{eq:disjoint-gluing}
		would imply $\supp(RX)\in\mathcal C$, contrary to the assumption
		on the input pair. Hence the output has prefix outside
		$\mathcal B$.

		Given $(w', s')$ let us now show how to recover $(w,s)$ uniquely, thus showing that this map is injective. First observe that, given $w'$, the integer  $s'$ determines the word
		$X$. In the remaining suffix $w'_{>s'}=RY$, take the shortest prefix whose
		support belongs to $\mathcal A \setminus \mathcal C$. This prefix
		is exactly $R$, since its support belongs to  $A\setminus\mathcal C$, and each of its shorter prefixes is
		the corresponding shorter prefix of the original word $w$, which
		does not belong to $\mathcal A \setminus \mathcal C$. We thereby recover $R$ and $Y$, and hence $(w,s)$. This also applies when $R$ or $X$ is empty.
	\end{proof}
	In the application to tight trees, we have the stronger
	inclusion $\mathcal C\subseteq\mathcal A\cap\mathcal B$:
	an embedding of the whole tree restricts to an embedding
	of either rooted subtree. However, the lemma requires
	only $\mathcal C\subseteq\mathcal A$.

	\begin{lemma}\label{lem:root-gluing}
		Suppose $T=T_1\cup T_2$, where
		\[
		V(T_1)\cap V(T_2)=e,\qquad e\in E(T_1)\cap E(T_2).
		\]
		Then
		\begin{equation}\label{eq:root-gluing}
			C(T_1,e)+C(T_2,e)\le M+C(T,e).
		\end{equation}
	\end{lemma}
	
	\begin{proof}
		Fix an ordering $e=(u_1,\ldots,u_r)$ and an assignment
		$u_i\mapsto a_i$ to distinct vertices with
		$\{a_1,\ldots,a_r\}\in E(H)$. Put
		\[
		A_0=\{a_1,\ldots,a_r\},\qquad U=V(H)\setminus A_0,
		\qquad m=n-r.
		\]
		A marked state with this root-edge assignment is of the form 
		$$((a_1, \ldots, a_{r-1}, P, a_r, Q), r+|P|).$$
		Consequently, for $w=(P, Q)$, these marked states are in bijection
		with all pairs $(w,s)$, where $w\in\Perm(U)$ and $0\le s\le m$:
		\begin{equation}\label{eq:fiber}
			(w,s)\longleftrightarrow
			\bigl((a_1,\ldots,a_{r-1},w_{\le s},a_r,w_{>s}),\ r+s\bigr).
		\end{equation}
		There are $(m+1)!$ such states. In particular, $s=0$ is allowed:
		the empty prefix outside $A_0$ corresponds to a marked state.

		Let $\mathcal A$, $\mathcal B$ and $\mathcal C$ be the families of
		sets $X\subseteq U$ for which $T_1$, $T_2$ and $T$, respectively,
		embed into $H[A_0\cup X]$ with the prescribed assignment on $e$.
		Clearly $\mathcal C\subseteq\mathcal A$, since if we restrict an embedding of $\mathcal C$ we get an embedding of $\mathcal A$.
		If $R\in\mathcal A$, $X\in\mathcal B$ and $R\cap X=\varnothing$,
		choose the two witnessing embeddings. They agree on every vertex
		of $e$ and have disjoint images outside $e$. Since
		$V(T_1)\cap V(T_2)=e$, their union is an injective embedding of
		$T$ into $H[A_0\cup R\cup X]$. Hence $R\cup X\in\mathcal C$.
		Consequently, $\mathcal A, \mathcal B, \mathcal C$ satisfy the hypothesis of Lemma~\ref{lem:prefix}.
		
		We now sum the conclusion of Lemma~\ref{lem:prefix} to obtain the conclusion of Lemma~\ref{lem:root-gluing}. For clarity of presentation, we make the counting explicit. Keep the ordering
		$e=(u_1,\ldots,u_r)$ fixed, and let
		$\mathcal I_H$ be the set of ordered edges of $H$, so that
		$|\mathcal I_H|=r!|E(H)|$.
		For each $\mathbf a=(a_1,\ldots,a_r)\in\mathcal I_H$, define
		\[
		\Omega_{\mathbf a}
		=
		\{(\pi,j)\in\Omega:
		v_i=a_i\text{ for }1\le i<r,\quad v_j=a_r\}.
		\]
		Let $C_{\mathbf a}(S,e)$ be the number of states in
		$\Omega_{\mathbf a}$ that are good for $(S,e)$.
		Every state $(\pi,j)\in\Omega$ determines the unique assignment
		$\mathbf a=(v_1,\ldots,v_{r-1},v_j)$.
		Consequently, the sets $\Omega_{\mathbf a}$ partition $\Omega$,
		and, for every $S$ under consideration,
		\begin{equation}\label{eq:assignment-sum}
			C(S,e)
			=
			\sum_{\mathbf a\in\mathcal I_H}C_{\mathbf a}(S,e).
		\end{equation}
		We count each good state once, regardless of how many embeddings satisfy its defining conditions.
		Now fix $\mathbf a\in\mathcal I_H$, and define $A_0$, $U$ and
		the families $\mathcal A,\mathcal B,\mathcal C$ as above.
		Under the bijection~\eqref{eq:fiber}, a pair $(w,s)$ corresponds
		to the state
		\[
		\pi=(a_1,\ldots,a_{r-1},
		w_1,\ldots,w_s,a_r,w_{s+1},\ldots,w_m),
		\qquad j=r+s.
		\]
		Its first $j$ vertices form exactly the set
		\[
		V_j(\pi)=A_0\cup\{w_1,\ldots,w_s\}.
		\]
		Moreover, the prescribed images of the root vertices are
		$u_i\mapsto a_i$ for every $1\le i\le r$.
		It follows that this state is good for $(T_1,e)$ if and only
		if $\{w_1,\ldots,w_s\}\in\mathcal A$. The corresponding
		equivalences hold for $(T_2,e)$ and $\mathcal B$, and for
		$(T,e)$ and $\mathcal C$. Therefore
		\[
		C_{\mathbf a}(T_1,e)=c_U(\mathcal A),\qquad
		C_{\mathbf a}(T_2,e)=c_U(\mathcal B),\qquad
		C_{\mathbf a}(T,e)=c_U(\mathcal C).
		\]
		Lemma~\ref{lem:prefix} consequently gives, for this assignment,
		\[
		C_{\mathbf a}(T_1,e)+C_{\mathbf a}(T_2,e)
		\le (m+1)!+C_{\mathbf a}(T,e).
		\]
		
		Finally, $m=n-r$ is the same for every assignment. Summing
		the preceding inequality over $\mathcal I_H$ and
		using~\eqref{eq:assignment-sum}, we obtain
		\[
		\begin{aligned}
			C(T_1,e)+C(T_2,e)
			&=
			\sum_{\mathbf a\in\mathcal I_H}
			\bigl(C_{\mathbf a}(T_1,e)+C_{\mathbf a}(T_2,e)\bigr)\\
			&\le
			\sum_{\mathbf a\in\mathcal I_H}
			\bigl((m+1)!+C_{\mathbf a}(T,e)\bigr)\\
			&=
			r!|E(H)|(m+1)!+C(T,e)\\
			&=
			M+C(T,e),
		\end{aligned}
		\]
		where the last equality follows from $m+1=n-r+1$
		and~\eqref{eq:counts}. This proves~\eqref{eq:root-gluing}.
		\iffalse
		Apply Lemma~\ref{lem:prefix} to these three families. Under
		\eqref{eq:fiber}, the three prefix counts are exactly the numbers
		of good states with the fixed assignment for $T_1$, $T_2$ and
		$T$. Summing~\eqref{eq:prefix-gluing} over all $r!|E(H)|$ possible
		assignments gives~\eqref{eq:root-gluing}, since
		$r!|E(H)|(m+1)!=M$ by~\eqref{eq:counts}.
		\fi
	\end{proof}
	
	\subsection{Completing the proof of Theorem~\ref{thm:main}}\label{subsec:induction}
	In this section we obtain an upper bound for $M$ by applying Lemmas~\ref{lem:leaf} and
	\ref{lem:root-gluing},  and use this to complete the proof of Theorem~\ref{thm:main}.
	
	\begin{theorem}\label{thm:rooted-count}
		For every tight $r$-tree $T$ with $t\ge1$ edges and every
		$e\in E(T)$,
		\begin{equation}\label{eq:rooted-count}
			M\le C(T,e)+(t-1)W.
		\end{equation}
	\end{theorem}
	
	\begin{proof}
		We proceed by induction on $t$. For $t=1$, the claim
		is~\eqref{eq:rooted-base}. Suppose $t\ge2$, and apply
		Lemma~\ref{lem:structure} at $e$. In case~\ref{case:leaf}, the smaller tight tree $S$ has $t-1$
		edges and for some edge $f \in E(S)$ the edge $e$ intersects $f$ in exactly $r - 1$ vertices. The induction hypothesis rooted at $f$, followed by
		Lemma~\ref{lem:leaf}, gives
		\[
		M\le C(S,f)+(t-2)W\le C(T,e)+(t-1)W.
		\]
		
		In case~\ref{case:gluing}, write $t_i=|E(T_i)|$. Then $t_i<t$
		for $i=1,2$ and $t_1+t_2=t+1$. Add the two induction inequalities
		and apply Lemma~\ref{lem:root-gluing} to obtain
		\[
		\begin{aligned}
			2M
			&\le C(T_1,e)+C(T_2,e)+(t_1+t_2-2)W\\
			&\le M+C(T,e)+(t-1)W.
		\end{aligned}
		\]
		Subtracting $M$ completes the induction.
	\end{proof}
	
	\begin{proof}[Proof of Theorem~\ref{thm:main}]
		If $E(H)=\varnothing$, the assertion is immediate. Otherwise
		$n\ge r$. Fix any edge $e$ of $T$. Since $H$ is $T$-free,
		$C(T,e)=0$. Theorem~\ref{thm:rooted-count} and~\eqref{eq:counts}
		give
		\[
		r!\,|E(H)|\,(n-r+1)! = M \le (t-1)W 
		= (t-1)(r-1)!\,|\sh H|\,(n-r+1)!.
		\]
		Cancelling $(r-1)!(n-r+1)!$ proves
		$r|E(H)|\le(t-1)|\sh H|$. Finally,
		$|\sh H|\le\binom{n}{r-1}$ gives~\eqref{eq:kalai-binomial}.
	\end{proof}


\begin{thebibliography}{100}
		
		\bibitem{ASS} M.~Ajtai, M.~Simonovits, E.~Szemerédi, \emph{Solution of the Erd\H{o}s-S\'{o}s Conjecture}, Presentation at Bondy Conference, Montreal, 2003.
		
		\bibitem{BPSS2019} G.~Besomi, M.~Pavez-Sign{\'e}, and M.~Stein, \emph{Degree conditions for embedding trees}, SIAM J. Discrete Math. \textbf{33} (2019), 1521--1555. 
		
		\bibitem{BPSS2021} G.~Besomi, M.~Pavez-Sign{\'e}, and M.~Stein, \emph{On the Erd{\H{o}}s--S{\'o}s conjecture for trees with bounded degree}, Combin. Probab. Comput. \textbf{30} (2021), 741--761. 
		
		\bibitem{BD} S.~Brandt and E.~Dobson, \emph{The Erd{\H{o}}s--S{\'o}s conjecture for graphs of girth 5}, Discrete Math. \textbf{150} (1996), 411--414. 
		
		\bibitem{DPRS} A.~Davoodi, D.~Piguet, H.~\v{R}ada, and N.~Sanhueza-Matamala, \emph{The asymptotic version of the Erd{\H{o}}s--S{\'o}s conjecture and beyond}, arXiv:2603.17755, 2026. 
		
		
		\bibitem{Epoch}
		T.~Adamczewski and T.~F.~Bloom,
		\emph{FrontierMath Erd\H{o}s}, Epoch AI report, September 2026.
		\url{https://epoch.ai/files/frontiermath-erdos.pdf}.
		
		\bibitem{Erdos}
		P.~Erd\H{o}s,
		\emph{Extremal problems in graph theory},
		in: Theory of Graphs and its Applications (M.~Fiedler, ed.),
		Academic Press, New York, 1965, 29--36.
		
		\bibitem{Erdos1964} P.~Erd{\H{o}}s, \emph{Extremal problems in graph theory}, in: Theory of Graphs and its Applications (Proc. Sympos. Smolenice, 1963), Publ. House Czechoslovak Acad. Sci., Prague, 1964, pp.~29--36. 
		
		\bibitem{EG}
		P.~Erd\H{o}s and T.~Gallai,
		\emph{On maximal paths and circuits of graphs},
		Acta Math. Acad. Sci. Hungar. \textbf{10} (1959), 337--356.
		
		\bibitem{ErdosSos1970} P.~Erd{\H{o}}s and V.~T.~S{\'o}s, \emph{Some remarks on Ramsey's and Tur{\'a}n's theorem}, in: Combinatorial Theory and its Applications, II (Proc. Colloq., Balatonf{\"u}red, 1969), Colloq. Math. Soc. J{\'a}nos Bolyai, Vol.~4, North-Holland, Amsterdam, 1970, pp.~395--404.
		
		\bibitem{FanSun2007} G.~Fan and L.~Sun, \emph{The Erd{\H{o}}s--S{\'o}s conjecture for spiders}, Discrete Math. \textbf{307} (2007), 3055--3062. 
		
		
		
		\bibitem{FF}
		P.~Frankl and Z.~F\"uredi,
		\emph{Exact solution of some Tur\'an-type problems}, J. Combin. Theory Ser. A \textbf{45} (1987), no.~2, 226--262.
		
		\bibitem{FJKMVtrunk}
		Z.~F\"uredi, T.~Jiang, A.~Kostochka, D.~Mubayi and J.~Verstra\"ete,
		\emph{Hypergraphs not containing a tight tree with a bounded trunk},
		SIAM J. Discrete Math. \textbf{33} (2019), no.~2, 862--873.
		\href{https://arxiv.org/abs/1712.04081}{arXiv:1712.04081}.
		
		\bibitem{FJKMVpaths}
		Z.~F\"uredi, T.~Jiang, A.~Kostochka, D.~Mubayi and J.~Verstra\"ete,
		\emph{Tight paths in convex geometric hypergraphs},
		Adv. Combin. \textbf{2020} (2020), Paper No.~1, 14~pp.
		\href{https://doi.org/10.19086/aic.12044}{doi:10.19086/aic.12044}.
		
		\bibitem{GKLO} S.~Glock, D.~K\"{u}hn, A.~Lo, D.~Osthus, \emph{The existence of designs via iterative absorption: hypergraph $F$-designs for arbitrary $F$}, Memoirs of the American Mathematical Society \textbf{284} (2023), monograph 1406).
		
		
		\bibitem{GoerlichZak2016} A.~Goerlich and A.~\.{Z}ak, \emph{On Erd{\H{o}}s--S{\'o}s conjecture for trees of large size}, Electron. J. Combin. \textbf{23} (2016), Paper~P1.52. 
		
		\bibitem{HRSW2020} F.~Havet, B.~Reed, M.~Stein, and D.~R.~Wood, \emph{A variant of the Erd{\H{o}}s--S{\'o}s conjecture}, J. Graph Theory \textbf{94} (2020), 131--158.
		
		
		\bibitem{Keevash}
		P.~Keevash,
		\emph{The existence of designs},
		\href{https://arxiv.org/abs/1401.3665}{arXiv:1401.3665}, 2014;
		revised 2024.
		
		\bibitem{KupitzPerles} Y. S. Kupitz, M. Perles, \emph{Extremal theory for convex matchings in convex geometric graphs}, Discrete Comput. Geom. \textbf{15}, (1996), 195--220. 
		
		
		\bibitem{McLennan2005} A.~McLennan, \emph{The Erd{\H{o}}s--S{\'o}s conjecture for trees of diameter four}, J. Graph Theory \textbf{49} (2005), 291--301.
		
		\bibitem{MubayiV2016} D.~Mubayi and J.~Verstra\"ete, \emph{Counting trees in graphs}, Electron. J. Combin. \textbf{23} (2016), no.~3, Paper P3.39.
		
		\bibitem{Pokrovskiy2024} A.~Pokrovskiy, \emph{Hyperstability in the Erd{\H{o}}s--S{\'o}s conjecture}, arXiv:2409.15191, 2024. 
		
		\bibitem{ReedStein2026} B.~Reed and M.~Stein, \emph{The Erd{\H{o}}s--S{\'o}s conjecture in dense graphs}, arXiv:2609.05417, 2026. 
		
		
		\bibitem{Rodl}
		V.~R\"odl,
		\emph{On a packing and covering problem},
		European J. Combin. \textbf{6} (1985), no.~1, 69--78.
		
		\bibitem{Rozhon2019} V.~Rozho\v{n}, \emph{A local approach to the Erd{\H{o}}s--S{\'o}s conjecture}, SIAM J. Discrete Math. \textbf{33} (2019), 643--664.
		
		\bibitem{SacleWozniak1997} J.-F.~Sacl{\'e} and M.~Wo{\'z}niak, \emph{The Erd{\H{o}}s--S{\'o}s conjecture for graphs without $C_4$}, J. Combin. Theory Ser. B \textbf{70} (1997), 367--372. 
		
		\bibitem{Stein} M.~Stein, \emph{Tree containment and degree conditions}, in: A.~M.~Raigorodskii and M.~T.~Rassias (eds.), Discrete Mathematics and Applications, Springer Optimization and Its Applications, Springer, 2020, pp.~459--486. 
		
		\bibitem{TinerTomlin2022} G.~Tiner and Z.~Tomlin, \emph{On the Erd{\H{o}}s--S{\'o}s conjecture for $k=9$}, Alabama J. Math. \textbf{45} (2022), 37--45.
		
		\bibitem{Wilson2025} C.~Wilson, \emph{A Tight Lower bound on Trees in Graphs}, https://arxiv.org/abs/2512.14890. 
		
		\bibitem{Wozniak1996} M.~Wo{\'z}niak, \emph{On the Erd{\H{o}}s--S{\'o}s conjecture}, J. Graph Theory \textbf{21} (1996), 229--234.
		
	\end{thebibliography}
\end{document}